\documentclass[10pt]{article}

\usepackage[usenames, dvipsnames]{color}

\usepackage{epsfig,amsmath,amsthm}
\usepackage{graphicx}
\usepackage{amssymb}
\usepackage{enumerate}
\usepackage{comment}
\usepackage{color}
\usepackage{paralist}
\usepackage[left=2cm,right=2cm,top=2cm,bottom=2cm]{geometry}
\usepackage{fancyhdr}
\usepackage{hyperref}
\hypersetup{
	colorlinks=true,
	linkcolor=blue,
	citecolor=red,
	urlcolor=blue,
	pdfborder={0 0 0}
}
\DeclareMathSizes{10}{10}{6}{4}

\usepackage{cleveref}
\crefname{theorem}{Theorem}{Theorems}
\crefname{thm}{Theorem}{Theorems}
\crefname{lemma}{Lemma}{Lemmas}
\crefname{lem}{Lemma}{Lemmas}
\crefname{remark}{Remark}{Remarks}
\crefname{rmk}{Remark}{Remarks}
\crefname{prop}{Proposition}{Propositions}
\crefname{notation}{Notation}{Notations}
\crefname{claim}{Claim}{Claims}
\crefname{defn}{Definition}{Definitions}
\crefname{definition}{Definition}{Definitions}
\crefname{cor}{Corollary}{Corollaries}
\crefname{example}{Example}{Examples}
\crefname{section}{Section}{Sections}
\crefname{figure}{Figure}{Figures}
\crefname{assumption}{Assumption}{Assumptions}

\newtheorem{theorem}{Theorem}

\newtheorem{claim}[theorem]{Claim}
\newtheorem{lemma}[theorem]{Lemma}
\newtheorem{cor}[theorem]{Corollary}
\newtheorem{prop}[theorem]{Proposition}
\newtheorem{definition}[theorem]{Definition}

\numberwithin{equation}{section}

\theoremstyle{definition}
\newtheorem{rmk}[theorem]{Remark}

\newcommand{\R}{\mathbb{R}}

\newcommand{\N}{\mathbb{N}}

\newcommand{\E}{\mathbb{E}}

\newcommand{\eps}{\varepsilon}

\newcommand{\B}{\mathbb{B}}
\newcommand{\D}{\mathbb{B}}

\newcommand{\ol}{\overline}

\newcommand{\Ub}{\mathbb{B}}

\begin{document}

\textheight = 21cm
\textwidth = 13cm

\title{\textsc{A characterisation of the continuum Gaussian free field in arbitrary dimensions.}}
\date{}
\author{Juhan Aru (EPFL) and Ellen Powell (Durham University)}
\maketitle
\vspace{-0.5cm}
\abstract{We prove that under certain mild moment and continuity assumptions, the $d$-dimensional continuum Gaussian free field is the only stochastic process satisfying the usual domain Markov property and a scaling assumption. Our proof is based on a decomposition of the underlying functional space 
in terms of radial processes and spherical harmonics.
}\bigskip \bigskip



\section{Introduction}
 The continuum Gaussian free field (GFF) is a generalisation of Brownian motion, taking the index set to higher dimensions, which is defined on {the unit ball} as follows. 

\begin{definition}[Gaussian free field]
Let $d \geq 1$ and $\D \subseteq \R^d$ denote the open unit ball. 

\noindent The $d$-dimensional zero boundary continuum GFF in $\D$ is the centred Gaussian process $(h^{\D},f)_{f\in C_c^\infty(\R^d)}$ whose covariance is given by 
$$\E((h^{\D},f)(h^{\D},g)) = \int\int_{(\R^d)^2} f(z)G^{\B}(z,w)g(w)\, dz dw;\quad f,g\in {C_c^\infty(\R^d)},$$
where $G^{\D}$ denotes the zero boundary Green's function for the Laplacian in $\D$. {Note in particular that $G^\D$ is identically zero in $\R^d\setminus \D$, so that $(h^\B,f)$ is almost surely zero as soon as $f$ has support outside of $\D$.}
\end{definition}

In dimension $d = 1$, this corresponds to a slightly awkward definition of the Brownian bridge on the interval $(-1,1)$: in this concrete one-dimensional case the GFF can actually be defined as an a.s.\ continuous function, i.e. indexed by points instead of test functions. In $d \geq 2$ the GFF does not make sense as a pointwise defined function, thus the more general definition. However, one can still heuristically see the continuum GFF as a Gaussian height function parameterised by a domain $D \subseteq \R^d$ \cite{SheffieldGFF}. The Gaussian free field has played many roles in probability and mathematical physics since the 1970s: it is the stationary solution of the stochastic heat equation; is used to describe both free and interactive Euclidean Quantum field theories; it appears as a corrector in stochastic homogenisation; to name just a few. Recently, the $2d$ GFF has played a crucial role in the probabilistic study of statistical physics models, Schramm-Loewner evolution, Liouville quantum gravity and Liouville field theory \cite{Dub, MS3, DS11, DKRV}. It is both the proven and the conjectured scaling limit of several natural discrete height functions \cite{NS, Kenyon_GFF, RiderVirag,BLR}

There are many characterisations of Brownian motion that single it out among $1d$ stochastic processes and it is similarly natural to also ask how the GFF in higher dimensions can be characterised. Recently, the $2d$ continuum GFF was characterised as the only conformally invariant field, satisfying a domain Markov property and certain minimal moment and continuity assumptions \cite{BPR18,BPR20}. 

This short article provides a first characterisation of the $d-$dimensional continuum Gaussian free field for $d \geq 3$, answering Question 6.1 in \cite{BPR18}. It also provides a new characterisation of the GFF in $d= 2$ that relaxes the assumption of conformal invariance and {thereby generalises the main result of \cite{BPR18} via a new {more elementary} proof}. Our characterisation is based only on {scaling} and the domain Markov property and in particular, we do not even need to assume rotational invariance.

More precisely, but still somewhat informally, we prove that, under certain mild moment and continuity assumptions, the $d$-dimensional Gaussian free field is the only random Schwartz distribution in $d \geq 2$ that satisfies the following domain Markov property - for each ball, the field inside can be written as a sum of an independent scaled and translated copy of the original field plus the ``harmonic extension" of its behavior on the boundary. We prove the theorem by identifying the covariance structure as the Green's kernel (lighter step), and then proving Gaussianity (more involved). The second step in particular uses a completely different argument to that in \cite{BPR18} - {instead of working with circle averages, we use a decomposition of $L^2(\D)$ based on spherical harmonics}. The two steps in our proof are quite independent of each other (although we use the knowledge of the covariance kernel to simplify some arguments later on), but both make strong use of the domain Markov property. Our proof, which does not even assume rotational invariance, stresses the nice interplay between the GFF and harmonic functions.

Recently, there has been a lot of interest in proving scaling limit results for random surface models, e.g., \cite{SquareIce,BLRsurface, GM21}. The characterisation given in this article {and its mild generalizations below} could potentially be helpful for obtaining such results, by giving a new way to identify the GFF as a continuum scaling limit. Importantly, one that does not rely on conformal nor rotational invariance - conditions that are both very difficult to prove from lattice considerations (though there is important recent progress \cite{mano}). \footnote{
In fact, let us thank here H. Duminil-Copin and V. Tassion for asking us whether only rotational invariance should suffice for the characterisation of the 2$d$ continuum GFF. This discussion happened in a nice and friendly workshop in Fribourg, organised by I. Manolescu, who we would hereby like to thank too. We would finally like to thank N. Berestycki, J.C. Mourrat, G. Ray and W. Werner for many interesting and fruitful discussions on this topic, {and two anonymous referees who pointed out several places where more clarity was required}. J.A.\ was supported by Eccellenza grant 194648 of the Swiss National Science Foundation.  } 

As mentioned, essentially the only property that is needed to characterise the continuum Gaussian free field is its domain Markov property. {It should be emphasised that, as stated, this domain Markov property also encapsulates a scaling property for the field (that could however be disposed of to some extent, see comments below).}

\begin{definition}[Domain Markov Property (DMP) {with scaling} for balls]\label{def:dmp}
Let $d \geq 2$ and $\D \subseteq \R^d$ denote the open unit ball. 
Suppose that $h$ is a random Schwartz distribution with support on $\D$, i.e. a random continuous functional $(h,f)$ indexed by $f \in C_c^\infty(\R^d)$ and zero as soon as $f$ has support outside of $\D$. 

 We define the random Schwartz distribution $h^{a+r\D}$ on domains $a + r\D$ as  the rescaled image of $h^\D$ under translation by $a$ and scaling\footnote{Note that if $h$ were a function and $(
    \cdot, 
    \cdot)$ was the $L^2$ inner product, this would be equivalent to $h^{a+r\B}(a+r\cdot)=r^{\frac{2-d}{2}}h^\B(\cdot)$, exhibiting half the scaling factor of $G^\B$.} by $r$: $$(h^{a+r\D},f(r\cdot+a))=r^{1+\frac{d}{2}}(h^\D,f(\cdot)).$$ 

We then say that $h$ satisfies the Domain Markov property if the following holds:
\begin{itemize}
\item Suppose that $a+r\D \subseteq \D$. Then 
   we can write $$(h^\D,f)_{f\in C_c^\infty(\R^d)}=(h^{a+r\D}_\D,f)_{f\in C_c^\infty(\R^d)}+(\varphi^{a+r\D}_\D,f)_{f\in C_c^\infty(\R^d)}$$ where the two summands are independent, $\varphi_\D^{a+r\D}$ is a stochastic process that a.s.\ corresponds to integrating against a harmonic function when restricted to $a+r\D$ and
   $h_\B^{a+r\B}$ is equal in law to a translated and scaled copy  $h^{a+r\D}$.
\end{itemize}
\end{definition}

We can now state the characterisation:

\begin{theorem}[Characterisation of the $d$-dimensional GFF]\label{thm_main}
Let $d \geq 2$ and $\D \subseteq \R^d$ denote the open unit ball. Suppose that $h$ is a random Schwartz distribution with support on $\D$, i.e. a random continuous functional $(h,f)$ indexed by $f \in C_c^\infty(\R^d)$ and zero as soon as $f$ has support oustide of $\D$.

If $h$ satisfies the following conditions, then for some constant $c$ the stochastic process $(h^{\D},f)_{f\in C_c^\infty(\R^d)}$ has the law of $c$ times a $d-$dimensional zero boundary Gaussian free field in $\D$.
\begin{enumerate}
    \item\label{dmp} \textbf{Domain Markov property {with scaling} for balls} as in \cref{def:dmp}.
          \item\label{mom} {\textbf{Moments} We have $\E((h^\D,f))=0$ and  $\E((h^\D,f)^4)<\infty$ for all $f\in 
        C_c^\infty(\R^d)$}. 
    \item\label{zbc} \textbf{Zero boundary conditions} For any sequence $(f_n)_{n\ge 0}$ of smooth positive functions with $\int f_n$ uniformly bounded, $d(\mathrm{supp}(f_n),0)\to 1$ as $n\to \infty$ and $\sup_n \sup_{r<1}\sup_{x,y\in \partial(r\B)} |f_n(x)/f_n(y)|<\infty$, we have that {$\E((h^\B,f_n)^2)\to 0$ as $n\to \infty$.} 
\end{enumerate}
\end{theorem}

We have not striven for most general technical assumptions, but rather have tried to keep the proofs light and self-contained. However, we believe that several assumptions can most likely be relaxed. Let us remark on those and the case $d = 1$.

\begin{itemize}
\item In fact, a similar characterisation holds also in $d = 1$ with basically the same proof. In this case one would just work directly in the space of continuous functions, and state the zero boundary condition pointwise.
\item We believe that the moment assumption can be probably relaxed using methods of \cite{BPR20}. 
\item Although using spherical harmonics is key in our proofs, one can also adapt the method to work for other reasonable domains, e.g. when one considers boxes instead of spheres both as the original domain and for the DMP. {Indeed, in this case one can still make sense of spherical harmonics defined on the boundary of the box, defined with respect to harmonic measure on the boundary as seen from the centre of the box. A characterisation of harmonic functions using such ``box averages'' will then give the required harmonicity for the covariance in the first part of the proof and everything in the second part will work the same way. 
More generally, a similar strategy ought to work for smooth convex domains, where there is a point of homothety inside the domain. The proof would, however, become more technical, and hence is not pursued here.} 
\item We use the same form of domain Markov property (DMP) as in \cite{BPR18}, but only for balls. The DMP plays a key role in our argument - it is used both in showing that the covariance kernel is harmonic off the diagonal, and in proving Gaussianity. With the current approach it seems difficult to relax the harmonicity condition for the extension of the boundary data, but it would be very interesting to determine if this is possible.
\item On the other hand, one can easily envisage replacing the condition of having an exact (scaled) copy of the field in the DMP with a more relaxed, martingale type of condition. {Namely, we could ask that $h^{\B}$ inside any sub-ball $a+r\B\subset \B$, can be decomposed as the independent sum of an a.s.\ harmonic function, and a random Schwartz distribution that has mean zero when tested against any test function}, {with some conditions on the second moment}. Indeed, this is known to suffice in the case of Brownian motion. We believe that our proof would adapt, as long as we imposed certain growth conditions for the variance of averages around a point, and certain homogeneity conditions. {For the sake of keeping things simple and short, we will not pursue this direction here.}
\item It might seem a bit surprising that we are not using rotational invariance. However, the reason is that the DMP already implies rather easily that the covariance is a harmonic function, and this is a very strong property. 
\item {Rather than assuming that $h$ is almost surely a Schwartz distribution, we could instead assume that: (1) $((h^\D,f))_{f \in C_c^\infty(\R^d)}$ is linear in $f$, and (2) the covariance $K_2(f,g):=\E((h^\D,f)(h^\D,g))$, is a continuous bilinear form on $C_c^\infty(\R^d)$\footnote{{with respect to the usual topology on $C_c^\infty(\R^d)$ in which $f_n\to f$ iff all derivatives of $f_n$ converge to the corresponding derivative of $f$, uniformly on the closure of ${\mathrm{supp}(f)}$}}. Indeed, we will use linearity of the process $(h^{\D},f)$ repeatedly, but will only use the fact that $f\mapsto (h^{\D},f)$ is a Schwarz distribution for justifying \eqref{K2limoutside} below, which would actually follow more directly had we assumed (2).} 
\item Finally, let us comment on the exact form of the zero boundary condition. In \cite{BPR18} this convergence was only asked for rotationally symmetric functions. We need the slightly generalised form (that is implied for example by conformal invariance and the rotationally invariant form) in only two places: to determine that the covariance is the Green's kernel and to obtain the uniqueness of the domain Markov property. Notice that for pointwise defined functions our condition is weaker than the usual pointwise zero boundary condition.
\end{itemize}

\begin{rmk}
When $d=2$ this is indeed a generalisation of the result in \cite{BPR18}. Namely, in \cite{BPR18} the authors assume that the law of a field $h^D$ is given for every simply connected domain $D$, and the family of laws satisfy conformal invariance (if $\varphi:D\to D'$ is conformal then the law of the pushforward of $h^D$ by $\varphi$, as a generalised function, is equal to that of $h^{D'}$) and a more general domain Markov property (whenever $D'\subset D$ we can write $h^D=h_D^{D'}+\varphi_D^{D'}$ as an independent sum, with $\varphi_D^{D'}$ harmonic in $D$ and $h_D^{D'}$ equal in law to $h^{D'}$). In particular, under these assumptions, $h^{\D}$ does satisfy \cref{dmp} of \cref{thm_main}. 

Therefore, if we take the assumptions of \cite{BPR18} as an input,\footnote{{also using the assumptions of \cite{BPR18} on $K_2$, which would equivalently work for our proof (as explained in the bullet point above)}} our result gives that $h^{\D}$ is equal in law to some multiple of the zero boundary Gaussian free field in $\D$. Then the assumption of conformal invariance identifies the law of $h^D$ for arbitrary simply connected $D$, and we reach the same conclusion as \cite{BPR18}.
\end{rmk}

We will now present the argument in three sections. First, we discuss some immediate consequences of the assumptions (along similar lines to \cite{BPR18}, so we will keep this brief). In Section \ref{sec:cov}, using the DMP, scaling and translation invariance, we show how to deduce that the covariance kernel is the Green's kernel (\cref{lem:cov}). Finally, the main part of the paper is Section \ref{sec:gauss}, where we prove Gaussianity (\cref{prop:gauss}) - we do this using solely the DMP, and a decomposition of the underlying functional space using spherical harmonics. Theorem \ref{thm_main} is an immediate consequence of \cref{lem:cov} and \cref{prop:gauss}. Sections \ref{sec:cov} and \ref{sec:gauss} can be read quite independently of each other, although we use the identification of the covariance kernel to simplify some arguments in the latter.

\begin{definition}[Scaling function]\label{def:s}
In what follows we write $s(r)$ for the function on $(0,\infty)$ defined by $-\log r$ when $d=2$ and $r^{2-d}$ when $d\ge 3$.
\end{definition}

\section{Immediate consequences}\label{sec:immediate_consequences}

{Here we discuss some immediate properties of our assumptions. The section is self-contained, but we remain brief, as similar properties have been shown in detail in the 2$d$ case in \cite{BPR18}.}

\paragraph{The domain Markov decomposition is unique.} Indeed, suppose that for some $a,r$ we had two  decompositions as in \ref{dmp} of \cref{thm_main}:  $$h^\B=h_\B^{a+r\B}+\varphi_\B^{a+r\B}=\tilde{h}_\B^{a+r\B}+\tilde{\varphi}_\B^{a+r\B}.$$ 
    Then for any $z\in a+r\B$, by harmonicity, there exists a sequence $(f_n)_{n\ge 0}$ of functions satisfying the zero boundary condition assumption \ref{zbc} from \cref{thm_main}, and such that $(\varphi,f_n(a+r\cdot))=\varphi(z)$ for all $n$ and any harmonic $\varphi$ in $a+r\B$. The zero boundary condition then implies that $(h_\B^{a+r\B}-\tilde{h}_\B^{a+r\B},f_n(a+r\cdot))\to 0$ a.s.\ along a subsequence as $n\to \infty$.  Thus, it must be the case that $\varphi_\B^{a+r\B}(z)=\tilde{\varphi}_\B^{a+r\B}(z)$ a.s. Applying this for a dense collection of $z$ and using harmonicity of $\varphi_\B^{a+r\B},\tilde{\varphi}_\B^{a+r\B}$ proves the uniqueness.
    
We will use the following consequences of this uniqueness repeatedly:
\begin{itemize}
   \item if $a+r\B \subset a'+r'\B\subset \B$ then \begin{equation}\label{eq:it_dmp}\varphi_\B^{a+r\B}-\varphi_\B^{a'+r'\B} \text{ is independent of } \varphi_{\B}^{a'+r'\B} \text{ and equal in law to } (r/r')^{\frac{2-d}{2}}\varphi_\B^{(a-a')+(r/r')\B};\end{equation}
    \item
   we can apply the domain Markov property in several balls at once. More precisely, if $B_1,\cdots, B_n$ are $n$ balls, and $h^\B=h^i+\varphi^i$ in each ball, then \begin{equation}
       \label{eq:multidmp}
 \varphi:=h^\B - \sum_{i=1}^n h^i \text{ is a.s.\ harmonic in } \cup_{i}B_i 
   \end{equation} and $h^i$ is independent of $\{\varphi,(h^j)_{j\ne i}\}$ for each $i$. Indeed, for any fixed $i$ we can write $\varphi=\varphi^i-\sum_{j\ne i}h^j$ and since $h^j$ is zero in $B_i$ for all $j\ne i$, $\varphi=\varphi^i$ is harmonic in $B_i$. But also for every $j\ne i$, $h^i$ is zero in $B_j$ and therefore $h^j$ is measurable with respect to $\varphi^i$. So the collection $\{\varphi,(h^j)_{j\ne i}\}=\{\varphi^i-\sum_{j\ne i} h^j,(h^j)_{j\ne i}\}$ is $\varphi^i$-measurable and therefore independent of $h^i$.
\end{itemize}

  \paragraph{The 2-point function.} For $z_1,z_2\in (1-\eps)\B$ with $|z_1-z_2|>2\eps$, define the harmonic functions $\varphi_\B^{z_1+\eps\B}, \varphi_\B^{z_2+\eps\B}$ according to the domain Markov decomposition of $h^\B$ in $z_1+\eps\B,z_2+\eps\B$ respectively. Then by \eqref{eq:it_dmp} and \eqref{eq:multidmp}, the quantity
 \[k_2(z_1,z_2)=\E(\varphi_\B^{z_1+\eps\B}(z_1)\varphi_\B^{z_2+\eps\B}(z_2))\] is well-defined, i.e., it does not depend on $\eps > 0$ satisfying the above conditions. Note that, by harmonicity, we can alternatively write \begin{equation}\label{k:conv}k(z_1,z_2)=\E((h^\B,\eta_{z_1})(h^\B,\eta_{z_2}))\end{equation} for any smooth functions { $\eta_{z_1}=\eta_{z_1}^\eps,\eta_{z_2}=\eta_{z_2}^\eps$ }of mass one, that are supported in $z_1+\eps\B,z_2+\eps\B$ and rotationally symmetric about $z_1,z_2$ respectively. 

 For every $\delta > 0$, the following crude upper bound for $z_1\ne z_2\in (1-\delta)\B$ is rather direct with $C = C(\delta)$:
 \begin{equation}\label{eq:var_bounds} |k_2(z_1,z_2)|\le C(1+s(|z_1-z_2|)). \end{equation}

 To justify this, observe that by Cauchy--Schwarz and the domain Markov property, it suffices to show that for any $z\in (1-\delta)\B$ and $\eps < \delta$
 \begin{equation}\label{eq:circlebound} \E(\varphi_\B^{z+\eps\B}(z)^2)\le C(1+s(\eps)).\end{equation}
 Indeed, as long as $d(z_i,\partial \D)>|z_1-z_2|/2$ for $i=1,2$ we can set  $\eps=|z_1-z_2|/2$ in the definition of $k_2(z_1,z_2)$; otherwise we crudely set $\eps = \delta/2$ (leading to the dependence of $C$ on $\delta$).
 
 Let us first consider the case $z=0$. If $2^{-n}\le \eps\le 2^{-(n-1)}$, then the domain Markov property implies that $\E(\varphi_\B^{\eps\B}(0)^2)\le \E(\varphi_\B^{2^{-n}\B}(0)^2)$. Further, by the domain Markov property again and also scaling, the right hand side can be written as the sum $\sum_{m=0}^{n-1} 2^{m(d-2)/2} X_m$, with the $X_m$ i.i.d.\ each having the law of  $\varphi_{\B}^{0.5\B}(0)$. Adding up the variances gives the desired bound.
  
 When $z\ne 0$ with $|z|=r<1$, let $r' = 1/(1+r)\in (1/2,1)$, $\eps' = \eps/(1+r)$ and $z'=-r'z$.
   By applying the Markov property for $h_\B^{z'+r'\B}$ in $\eps'\B$ we can write $$\varphi_\B^{\eps'\B}(0)  = \varphi_\B^{z' +r'\B}(0) +  \tilde{\varphi}(0)$$
   where the summands are independent and by translation invariance and scaling, $\tilde{\varphi}(0)$ has the law of $(r')^{2-d}$ times $\varphi_\B^{z+\eps\B}(z)$. But since the variance of the summands add up, the variance of $\tilde{\varphi}(0)$ is no greater than the variance of $\varphi_\B^{\eps'\B}(0)$. Hence the claim follows from the case $z = 0$.

\paragraph{The 4-point function.}
Similarly, due to \eqref{eq:multidmp}, for $z_1,\dots, z_4\in \B$ with $\min_{j\ne i}{|z_i-z_j|}>2a_i$ and $d(z_i,\partial \B)>a_i$ for each $i$, $k_4(z_1,\dots, z_4)=\E(\prod \varphi_\B^{z_i+a_i\B}(z_i))$ is well-defined, i.e., it does not depend on the choice of $a_i$ satisfying the above conditions. Again we can write this as $\E(\prod_{i=1}^4(h^\B,\eta_{z_i}))$ for any smooth functions $(\eta_{z_i})_{1\le i \le 4}$, with $\eta_{z_i}$ having mass one, being rotationally symmetric about $z_i$ and supported in $z_i+a_i\B$ for $i=1,2,3,4$. 

Using the same argument as in the 2-point case, one can now bound $|k_4(z_1, \dots, z_4)|$. Indeed, by Holder's inequality it suffices this time to obtain bounds on the fourth moments of $\varphi_\B^{z+\eps\B}(z)$ for $\eps>0$. This can be done very similarly to the 2-point bound. One first treats $z = 0$ using the i.i.d nature of increments - the only change here is that when expanding $\E((\sum_{m=0}^{n-1} \alpha_m X_m)^4)$ with independent $X_m$s one must consider both $\sum \alpha_m^4 \E(X_m^4)$ and $\sum_{m\ne l} \alpha_m^2\alpha_l^2 \E(X_m^2)\E(X_l^2)$ - and then transports this bound to general $z$ using translation invariance and scaling. The result is that for any given $\delta>0$, there exists $C=C(\delta)$ such that \begin{equation}\label{eq:circlebound2}
\mathbb{E}((\varphi_\B^{z+\eps\B}(z))^4)\le C(1+s(\eps)^2)\end{equation} for any $z\in (1-\delta)\B$ and $\eps<\delta$ (we omit the details). This leads to the  final conclusion: there are constants $C(\delta)>0$ such that for any distinct $z_1,z_2,z_3,z_4\in (1-\delta)\B$ \begin{equation}\label{eq:moment_bounds}  |k_4(z_1,\dots, z_4)|^4\leq C(\delta)\prod_{i=1}^4 (1+\max_{j\ne i} s(|z_i-z_j|)^2).\end{equation}

\paragraph{The diagonal contribution of the 2-point function is negligible.}

We can now show that  $k_2$ is the covariance function of the field, in the sense that for any  $f_1,f_2 \in C_c^\infty(\B)$, if  we have 
\begin{align}\label{Kk}K_2(f_1,f_2)& := \mathbb{E}((h^{\D},f_1)(h^\D,f_2)) \nonumber \\ 
& =\iint_{\B^2} f_1(z_1)f_2(z_2)k_2(z_1,z_2) \, dz_1 dz_2,\end{align}
where the right-hand side is well defined as the limit
\begin{equation}\label{eq:Kklimit}\lim_{a\to 0}\iint_{\B^2\setminus\{|z_1-z_2|\le a\}} f_1(z_1)f_2(z_2)k_2(z_1,z_2) \, dz_1 dz_2\end{equation}
(see just below for a proof). In other words, we show that $K_2$ is an integral kernel with density $k_2$, and that there is no ``diagonal contribution'' of $K_2$, so it does not matter than $k_2$ is only defined off the diagonal. This already rules out, for example, the possibility that the field is a white noise or one of its derivatives.\footnote{We would like to thank A. Sep\'{u}lveda for raising the need for extra clarity on this point.}

Let us now justify \eqref{Kk}. {For each $w\in \D$ and $0<\eps<d(w,\partial \B)$, let  $\eta_{w}^\eps$ be a unit-mass radially symmetric mollifier in the $\eps$ ball around $w$, as in \eqref{k:conv}.}
Then if $f_1,f_2\in C_c^\infty(\B)$ {(we fix these from now on)}, 
the functions {$$f_i^\eps(z):= \int_{\B} f_i(w)\eta_{z}^\eps(w) \, dw =\int_{\B} f_i(w) \eta_w^\eps(z) \, dw$$} converge to $f_i$ in $C_c^\infty(\D)$ for $i=1,2$. {The assumption that $h^\D$ is almost surely a Schwartz distribution therefore implies that \begin{equation}\label{K2liminside}K_2(f_1,f_2)=\mathbb{E}(\lim_{\eps\to 0} (h^{
\D},f_1^\eps)(h^{\B},f_2^\eps))\end{equation}
where the limit inside the expectation is almost sure. Moreover, $(h^{\D},f_1^\eps)(h^{\D},f_2^\eps)$ is uniformly bounded in $L^2$. To see this, we bound \begin{align*}\mathbb{E}((h^{\D},f_1^\eps)^2(h^{\D},f_2^\eps)^2) & = \iint f_1(u)f_1(v)f_2(w)f_2(x)\mathbb{E}((h^{\D},\eta_u^\eps)(h^{\D},\eta_v^\eps)(h^\D,\eta_w^\eps)(h^\D,\eta_x^\eps)) \, du dv dw dx \\ & \le \sup_{i=1,2}\sup_{z\in \B} |f_i(z)|  \iint |\mathbb{E}((h^{\D},\eta_u^\eps)(h^{\D},\eta_v^\eps)(h^\D,\eta_w^\eps)(h^\D,\eta_x^\eps))| \, du dv dw dx,\end{align*} 
where the integral on the right of the final display is uniformly bounded in $\eps$ \footnote{We would like to thank G. Woessner for pointing out the obvious problem with the argument claiming to show this in the published version.}. Indeed, by \eqref{eq:multidmp}, if $u,z,s,w$ are all at distance $\ge 2\eps$ from one another, we have $|\mathbb{E}((h^{\D},\eta_u^\eps)(h^{\D},\eta_v^\eps)(h^\D,\eta_w^\eps)(h^\D,\eta_x^\eps))|=|k_4(u,v,w,x)|$, and the integral of $|k_4|$ over this region is uniformly bounded in $\eps$ using \eqref{eq:moment_bounds} (see below \eqref{eq:k4sphere_crude} for a more detailed explanation). If three or more of the points $u,v,w,x$ are within distance $2\eps$ of one another, we can use Cauchy--Schwarz, the domain Markov property and scaling to bound $|\mathbb{E}((h^{\D},\eta_u^\eps)(h^{\D},\eta_v^\eps)(h^\D,\eta_w^\eps)(h^\D,\eta_x^\eps))|\le \sup_{z\in \B}\E((h^\B,\eta_{z}^\eps)^4)=
\sup_{z\in \B} \E((\varphi_\B^{z+2\eps\B}(z)+(h_\B^{2\eps\B},\eta_z^\eps))^4)\le C(1+s(\eps)^2)$ for $C$ uniform over $\eps$ (small enough). This region has area of order $\eps^{2d}$, while $s(\eps)^2$ is either logarithmic or equal to $\eps^{4-2d}$, so the integral over this region is also uniformly bounded (in fact going to $0$) in $\varepsilon$. Finally, if exactly two of the points (say $u$ and $v$) are within distance $2\eps$, and so $d_x,d_w>2\eps$ where $d_x:=\max\{|x-u|,|x-v|,|x-w|\}$ and $d_w:=\max\{|w-u|,|w-v|,|w-x|\}$, similar reasoning gives $|\mathbb{E}((h^{\D},\eta_u^\eps)(h^{\D},\eta_v^\eps)(h^\D,\eta_w^\eps)(h^\D,\eta_x^\eps))|\le C \sqrt{1+s(\eps)^2}\sqrt{(1+|s(d_x)|)}\sqrt{1+|s(d_w)|}$. Now, the integral over $x,w$ is again uniformly bounded because of the size of the singularity, and the integral over $u, v$ has area $\eps^{d}$; thus we obtain uniform boundedness in $\eps$ and conclude this case too. 
Thus the family $(h^{\D},f_1^\eps)(h^{\D},f_2^\eps)$ is uniformly integrable and we can take the limit outside of the expectation in \eqref{K2liminside}, i.e., we have \begin{equation}\label{K2limoutside} K_2(f_1,f_2)=\lim_{\eps\to 0} \mathbb{E}((h^{
\D},f_1^\eps)(h^{\B},f_2^\eps)).\end{equation} 

We are now in a good position to prove \eqref{Kk}, but still need to deal with the contribution to $K_2$ from points ``near the diagonal''. For this, we break up the right-hand side of \eqref{K2limoutside}} as 
$$\lim_{\eps\to 0} \left(\iint_{|z_1-z_2|>2\eps}f_1(z_1)f_2(z_2)K_2(\eta_{z_1}^\eps, \eta_{z_2}^\eps)\, dz_1 dz_2 + \iint_{|z_1-z_2|\le 2\eps}f_1(z_1)f_2(z_2)K_2(\eta_{z_1}^\eps, \eta_{z_2}^\eps) dz_1 dz_2\right)$$
where by
\eqref{k:conv}, we have $$\iint_{|z_1-z_2|>2\eps}f_1(z_1)f_2(z_2)K_2(\eta_{z_1}^\eps, \eta_{z_2}^\eps)\, dz_1 dz_2 = \iint_{|z_1-z_2|>2\eps}f_1(z_1)f_2(z_2)k_2(z_1,z_2)\, dz_1 dz_2$$ for each $\eps.$ On the other hand, $K_2(\eta_{z_1}^\eps, \eta_{z_2}^\eps)^2\le {\E((h^\B,\eta_{z_1}^\eps)^2)\E((h^\B,\eta_{z_1}^\eps)^2)}$ for each $\eps$ by Cauchy-Schwarz, where by  \eqref{eq:circlebound}, the domain Markov property and scaling, we can bound $\E((h^\B,\eta_{z}^\eps)^2)=\E(\varphi_\B^{2\eps\B}(z)^2)+\E((h_\B^{2\eps\B},\eta_z^\eps)^2)\le C(1+s(\eps))$ for some constant $C$. Note that this constant $C$ does not depend on $\eps$, and the bound holds for all $z$ in the compact supports of $f_1$ and $f_2$ simultaneously. This implies that $$\limsup_{\eps\to 0} \iint_{|z_1-z_2|\le 2\eps}f_1(z_1)f_2(z_2)K_2(\eta_{z_1}^\eps, \eta_{z_2}^\eps) dz_1 dz_2=0$$ and hence that $$\lim_{\eps\to 0} \iint_{|z_1-z_2|>2\eps}f_1(z_1)f_2(z_2)k_2(z_1,z_2)\, dz_1 dz_2 \text{ exists and is equal to } K_2(f_1,f_2)$$ as required. A similar argument appears in  \cite[Proof of Lemma 2.18]{BPR18}.

\paragraph{Spherical averages.} 
For $z\in \B$ and $\eps<d(z,\partial \B)$ we define $h_\eps(z)=\varphi_\B^{z+\eps\B}(z)$ to be the $\eps$-\emph{spherical average} of $h^\B$ around $z$. 
{
The spherical average admits the following natural approximations. Write $\rho_n^{\eps}$ for a sequence of smooth  test functions with total mass one that are rotationally symmetric about $z$ and supported in the annular region $z+\{(1-2^{-2n})\eps\B\setminus (1-2^{-n})\eps\B\}$ for each $n$. Then $\E((h^\B,\rho_n^{\eps})^2)=\E(h_\eps(z)^2)+\E((h_\B^{z+\eps\B},\rho_n^{\eps})^2)$ for each $n$ by the Markov property, where the second term on the right-hand side goes to $0$ as $n\to \infty$ by the zero boundary condition assumption. It therefore follows that $$\E(h_\eps(z)^2)=\lim_{n\to \infty} \E((h^\B,\rho_n^{\eps})^2)$$
and moreover this convergence is uniform for, say, $\eps \in (\delta,1)$.

Using the same construction as above with $z=0$ and $\eps=r_n$ converging to $1$ as $n\to\infty$, we also see that $\lim_{n\to \infty} \E(h_{r_n}(0)^2)=\lim_{n\to \infty} \E(h^\B,\rho_n^{r_n})^2)$, and this final limit is $0$, again by the zero boundary condition. Since $\E(h_r(0)^2)$ is decreasing as $r\uparrow 1$, this implies that}
\begin{equation}\label{eq:circav0bc}
\E(h_r(0)^2)\searrow 0 \text{ as } r\uparrow 1.
\end{equation}

\section{Covariance is the Green function}\label{sec:cov}
Let us start by showing that scaling and translation invariance together with the domain Markov property already imply that the covariance kernel is the Green's function:

\begin{prop}\label{lem:cov}
The function $k_2(x,y)$ (defined for $x\ne y)$ is a positive multiple of the zero boundary Green's function. In particular, by \eqref{Kk}, this implies that the bilinear form $K_2$ is a multiple of the map $(f_1,f_2)\mapsto \iint_{\B^2} G^\B(z_1,z_2)f_1(z_1)f_2(z_2) \, dz_1 dz_2$. 
\end{prop}

\begin{proof} 
Write $G^\B$ for the zero boundary Green's function in $\B$. 
We are going to use the following characterisation of $G^\B$ (see for example \cite[Lemma 3.7]{WPgff}). 
\begin{itemize}
    \item[$\star$] Suppose that for $y\in \B$, $k_y(x)$ is a harmonic function defined in in $\B\setminus \{y\}$, such that $k_y(x)-bs(|x-y|)$ is bounded in a  neighbourhood of $y$ for some $b > 0$ and such that $(k_y,f_n)
    \to 0$ as $n\to \infty$ for any sequence of functions $f_n$ as in our zero boundary condition \ref{zbc} of \cref{thm_main}.\footnote{Note that for any fixed $y$, the support of $f_n$ will not intersect $y$ and so $(k_y,f_n)$ makes perfect sense as the integral of $k_y$ against $f_n$ in $\B$.}\footnote{The proof in \cite{WPgff} works exactly the same if we use this ``zero boundary condition'' for $k_y$ rather than a pointwise zero boundary condition.} Then $k_y(x)=bG^\B(x,y)$ for all $x\ne y$; $x,y\in \B$.
\end{itemize}

Let us first see that for fixed $y\in \B$ the function $k_y(\cdot):=k_2(y,\cdot)$ is harmonic away from $\{y\}$. {For this, suppose that $\eta>0$ and $x\in \B$ is such that $|x-y|\wedge d(x,\partial \B)>\eta$. Then if we choose $a$ such that $|x-y|>\eta+2a$, we have $k_2(x,y)=\mathbb{E}(\varphi_{\B}^{x+(\eta+a)\B}(x)\varphi_{\B}^{y+a\B}(y))$ by definition, where $\varphi_{\B}^{x+(\eta+a)\B}$ is almost surely harmonic in $x+\eta\B$. This means that $\varphi_{\B}^{x+(\eta+a)\B}$ satisfies $\varphi_{\B}^{x+(\eta+a)\B}(x)=\int \varphi_{\B}^{x+(\eta+a)\B}(w) \rho_x^\eta(dw) $ (the mean value property) with probability one, where $\rho_{x}^\eta$ is uniform measure on $\partial(x+\eta\B)$. Moreover, the domain Markov decomposition (and its uniqueness) mean that $\varphi_{\B}^{w+(a/2)\B}(w)=\varphi_{\B}^{x+(\eta+a)\B}(w)+Z$ with $Z$ centered and independent of $\varphi_{\B}^{y+a\B}(y)$, so again using the definition of $k_2$, we have that $k_2(w,y)=\mathbb{E}(\varphi_{\B}^{x+(\eta+a)\B}(w)\varphi_{\B}^{y+a\B}(y))$ for each $w\in \partial(x+\eta\B)$. The conclusion is that for any given $\eta>0$, and any $x\in \B$ such that $|x-y|\wedge d(x,\partial \B)>\eta$:}
\begin{equation} \label{eqn:mvpk2} k_2(x,y)=\int k_2(w,y) \rho_{x}^\eta(dw).\end{equation}
Note that the support of $\rho_x^\eta$ does not include $y$ by assumption, so the integral against $k_2(\cdot, y)$ is perfectly well defined. (The proof {above} is verbatim that given in \cite[Lemma 2.9]{BPR18}, which applies directly to all $d\ge 2$.) \eqref{eqn:mvpk2} implies that $k_y$ is indeed harmonic in $\B\setminus \{y\}$. Note that in particular we have continuity in this region, which we did not assume a priori. 

 {We now check the boundary condition for $k_y$. If $f_n$ are a sequence of functions as in \ref{zbc} of \cref{thm_main}, we have $(k_y,f_n)=\int k_2(y,x)f_n(x) \, dx$ by defintion of $k_y$. Then, by the harmonicity shown above (and because $f_n$ has support outside of some fixed neighbourhood $B_\delta(y)$ of $y$ for $n$ large enough), this integral is equal to $\int k_2(w,x) f_n(x) F(w) \, dx dw$ where we can take $F$ to be a fixed smooth function supported in $B_\delta(y)$ that is radially symmetric about $y$. Due to \eqref{Kk}, we see that this is equal to $K_2(F,f_n)=\E((h^{\B},F)(h^{\B},f_n))$, which is in turn bounded (using Cauchy--Schwarz) by the square root of $\E((h^{\B},F)^2)\E((h^\B,f_n)^2)$. Applying the boundary condition \ref{zbc}, we see that this indeed converges to $0$ as $n\to \infty$.}

Now, notice that by \eqref{eq:var_bounds}, for any $\delta > 0$ there are some constants $c_1, c_2 > 0$ such that $w\mapsto k_2(0,w) + c_1 s(|w|)+c_2$ is a positive harmonic function in $(1-\delta)\B \setminus \{0\}$. Thus by B\^ocher's theorem \cite[Chapter III]{axlerharm}, we conclude that {there is some harmonic function $\nu(w): (1-\delta)\B \to \R$ such that in $(1-\delta)\B \setminus \{0\}$ we can write $k_2(0,w) + c_1 s(|w|)+c_2  = c_3 s(|w|) + \nu(w)$.} In particular, there is some $b > 0$ such that $k_2(0,w) - bs(|w|)$ is harmonic and bounded in $(1-2\delta)\B\setminus\{0\}$ and thus can be extended to a harmonic function on $(1-2\delta)\B$. Note that $b$ must be positive, since $\int k_2(0,z)\rho_0^{|w|}(dz)=\E(h_{|w|}(0)^2)$ is positive and increasing to $\infty$ as $w \downarrow 0$ by the domain Markov property. 

From here, we return to a fixed $y\in \B$. We choose $\eps>0$ such that $d(y,\partial \B)>2\eps$ and write $h^\B|_{y+2\eps\B}=h_\B^{y+2\eps\B}+\varphi_\B^{y+2\eps\B}$. Then for $x\in y+\eps\B$ we have $k_2(x,y)=\E((h^\B,\eta_{x})(h^\B,\eta_{y}))$ where $\eta_x$ and $\eta_y$ are smooth functions with mass one, radially symmetric around $x,y$ and supported in small non-intersecting balls around $x$ and $y$ respectively. Using the decomposition and harmonicity of $\varphi_\B^{y+2\eps\B}$ we have $$k_2(x,y)=\E((h^\B_{y+2\eps\B},\eta_x)(h^\B_{y+2\eps\B},\eta_y))+\E(\varphi_\B^{y+2\eps\B}(x)\varphi_\B^{y+2\eps\B}(y))$$ 
where the second term on the right-hand side is bounded by Cauchy-Schwarz and \eqref{eq:var_bounds}. The first term is equal to $bs(|x-y|)+O(1)$ in some neighbourhood of $y$ 
by the definition of the distribution of $h^{y+2\eps\B}_\B$,  and the previous paragraph.

We have therefore shown that $x\mapsto k_2(x,y)$ satisfies the condition $\star$ and is therefore equal to a multiple (which must actually be $b$) of $G^\B(x,y)$. 
Since $y\in \B$ was arbitrary we are done.
\end{proof}

\begin{cor}\label{cor}
For any $f\in H^{-1}(\B)$, take a sequence of smooth functions $(f_n)_{n\ge 0}$ such that $G^\B(f-f_n,f-f_n)\to 0$. Then we may define $(h^\B,f)$ to be the $L^2$ limit of $(h^\B,f_n)$ as $n\to \infty$. The limit does not depend on the approximation.
\end{cor}

In what follows, we will therefore use the notation $(h^\B, f)$ for $f\in H^{-1}(\B)$ without further justification.

\section{Gaussianity}\label{sec:gauss}

It now remains to argue that the field is Gaussian. We do this via a decomposition of the field into a sequence of radial processes. We show using the domain Markov property that each of these processes is Gaussian, and that moreover, the whole sequence is jointly Gaussian.

First, we will see a bound for the 4-point function that is reminiscent of Wick's theorem: this will help us deduce continuity of our processes. In the second subsection, we do the basic case - the case of spherical averages. Finally, we extend this to a wider range of processes, obtained from so called spherical harmonics. The Gaussianity of the field $(h^\B, f)_{f \in C_c^\infty(\R^d)}$ is proved in this final subsection.

\subsection{Weak Gaussianity in terms of the 4-point function}

To prove continuity of our radial processes, we need the following bound that is implied by our assumption on existence of fourth moments. Notice that this can be seen as establishing a very weak form of Wick's theorem, i.e. getting us closer to Gaussianity. Recall that for $r>0$, $h_r(0)=\varphi_\B^{r \B}(0)$ is the spherical average at radius $r$ around the origin, and converges to $0$ in probability as $r\uparrow 1$.
\begin{lemma}\label{lem:4mtbnd} For $r\in (0,1)$ we have that 
\begin{equation}\label{4mom_av} \E(h_r(0)^4)=\iint_{\partial(r\B)^4} k_4(z_1,z_2,z_3,z_4) \prod_{i=1}^4 \rho_0^r(dz_i) \end{equation} 
where the right-hand side is well defined as the limit when $a\to 0$ of $\iint_{|z_i-z_j|>a \, \forall i,j} k_4(z_1,z_2,z_3,z_4) \prod \rho_0^r(dz_i).$ Moreover, for some constant $c(d)$ and some $\eta\in[0,1)$:
\begin{equation}\label{K4bound} |k_4(z_1,z_2,z_3,z_4)|\le c(d) \delta^{-\eta} g(z_1,z_2,z_3,z_4)\end{equation}  for all $\delta\in (0,1]$ and distinct $z_1,\dots, z_4 \in \partial (1-\delta)\B$, where $$g(z_1,z_2,z_3,z_4)= G^\B(z_1,z_2)G^\B(z_3,z_4)+G^\B(z_1,z_3)G^\B(z_2,z_4)+G^\B(z_1,z_4)G^\B(z_2,z_3)$$
{is the four-point function for the $d-$dimensional zero boundary GFF in $\B$.}
\end{lemma}

Before proving the lemma, let us show that for any {$0<r<1-\eta$}, 
and for some constant $C(\eta)$ depending only on $\eta$,
\begin{equation}\label{eq:k4sphere_crude}
\limsup_{a>0}\iint_{\partial(r\B)^4\setminus\{A_a\}} |k_4(z_1,z_2,z_3,z_4)| \prod_{i=1}^4 \rho_0^r(dz_i) \le C(\eta) s(r)^2,
\end{equation}
where $A_a=\{z_1,z_2,z_3,z_4\in \B^4: d(z_i,z_j)\le a \text{ for some } i\ne j\}$.
 Indeed, by symmetry it suffices to bound the integral over the region where $\min(|z_i-z_j|)=|z_1-z_2|$. However, on this region, by \eqref{eq:moment_bounds}, we can bound $|k_4|$ above by an $\eta$-dependent constant times
$$ (1+s(|z_1-z_2|))\sum_{i\ne 3} \sqrt{1+s(|z_i-z_3|)} \sum_{j\ne 4} \sqrt{1+s(|z_j-z_4|)}.$$ Then expanding the product of the two sums, we can show the desired bound for the integral of each term separately (uniformly in $a$), using Cauchy--Schwarz in the integral over $z_3$ and $z_4$ for all terms, except $s(|z_3-z_4|)s(|z_1-z_4|)$ which can be integrated directly. Here we are using the fact that for any fixed $z_1\in \partial (r\B)$, we have that $\int_{\partial (r\B)} s(|z_1-z|){\rho_0^r(dz)}=O(s(r))$. 
In particular, applying dominated convergence, 
this shows that the integral in \eqref{4mom_av} is well-defined and finite.

The expression for the 4-point function now also follows from dominated convergence:

\begin{proof}[Proof of \eqref{4mom_av}] 
Fix $r\in (0,1)$ and for each $n\ge 1$ partition the sphere $\partial((r-2^{-n})\B)$ into regions each having diameter no larger than $2^{-n}$. For each of these regions we can then choose a ball of maximal radius centred at a point in the region, so that the ball intersected with the sphere lies inside the region, but within distance $2^{-2n}$ from the boundary of the region. This produces a sequence $\{z_{i,n}\}_{1\le i\le m_n}$  and radii $
\{r_{i,n}\}_{1\le i \le m_n}$ (all less than $2^{-(n+1)}$) such that the balls $z_{i,n}+r_{i,n}\B$ do not intersect. 

We can now set $\nu_{i,n}$ for each $i,n$ to be a smooth mollifier supported on $z_{i,n}+r_{i,n}\B$, that is radially symmetric about $z_{i,n}$ and has total mass one. Using that $k_2=bG^\B$, we see that $m_n^{-1}\sum_{i=1}^n (h^\B,\nu_{i,n})\to h_r(0)$ in $L^2$ as $n\to \infty$. On the other hand, the definition of $k_4$ gives that $k_4(z_{i,n},z_{j,n},z_{k,n},z_{l,n})=\E((h^\B,\nu_{i,n})(h^\B,\nu_{j,n})(h^\B,\nu_{k,n})(h^\B,\nu_{l,n}))$ for any distinct $1\le i,j,k,l\le m_n$ and every $n$. Dominated convergence using \eqref{eq:k4sphere_crude} then allows us to conclude. 
\end{proof}

The Wick-type of bound is slightly trickier:

\begin{proof}[Proof of \eqref{K4bound}]  Fix $z_1,z_2,z_3,z_4$ distinct and for each $j$ write $a_j=\min_{i\ne j} d(z_i,z_j)/2$. 

Denote $B_j:=(z_j+a_j\B)\cap \B$ for each $j$ so that the $B_j$ do not intersect. Write $\rho_j$ for uniform measure on $\partial(z_j+a_j\B)$ and $\tilde{\rho}_j$ for harmonic measure on $\partial B_j\setminus \partial \B$, seen from $z_j$. 

\begin{claim}\label{claim:4mt} We can write
\begin{equation}\label{eq:k4rewrite}k_4(z_1,z_2,z_3,z_4)=\E((h^\B,\tilde{\rho}_1)(h^\B,\tilde{\rho}_2)(h^\B,\tilde{\rho}_3)(h^\B,\tilde{\rho}_4)).\end{equation} 
\end{claim}
\noindent This claim could be checked straightforwardly if we could apply the domain Markov property inside the regions $B_1$ to $B_4$, so it should not be too surprising. However, since the Markov property has only been assumed for balls, we will have to do a little bit of work to show this. This will be carried out shortly, but let us first see how \eqref{K4bound} follows. 

First, notice that using the domain Markov property we can write $$\E((h^{2\B},\tilde\rho_j)^4)=\E((\varphi^{\B}_{2\B},\tilde{\rho}_j)^4)+6\E((\varphi^{\B}_{2\B},\tilde{\rho}_j)^2)\E((h^\B_{2\B},\tilde{\rho}_j)^2)+\E((h^\B_{2\B},\tilde{\rho}_j)^4)$$ and hence bound $\E((h^\B,\tilde{\rho}_j)^4)\le \E((h^{2\B},\tilde{\rho}_j)^4)$ for each $j$ . By the same dominated convergence argument as for \eqref{K4bound}, and writing $k_4^{2\B}$ for the four-point function of $h^{2\B}$, we see that the latter is equal to $\iint_{\partial (z_j+a_j\B)} |k_4^{2\B}(x,y,z,w)| \tilde{\rho_1}(dx)\dots \tilde{\rho_4}(dw)$. But then by \eqref{eq:k4sphere_crude} and using that each $z_j+a_j\B$ is far from the boundary of $2\B$ we see that this is less than some constant, not depending on $\delta$, times
\begin{equation}\label{density_sup}\sup_{\partial B_j\setminus \partial \B} \|\frac{d\tilde{\rho}_j}{d\rho_j}\| \; s(a_j)^2.
\end{equation}In the case that $a_j>\delta$, $\sup_{\partial B_j\setminus \partial \B} \|d\tilde{\rho}_j/d\rho_j\|$ can be bounded above by the probability that $d$-dimensional Brownian motion on the half space $\{(x_1,\dots, x_d): x_1>0\}$, started from $(\delta/a_j)$ reaches the boundary of the unit sphere before hitting the hyperplane $\{x_1=0\}$. This is bounded above by a constant times $(\delta/a_j)$  \cite[Theorem 4.4]{Burdzy}. When $\delta>a_j$ we have {$\sup_{\partial B_j\setminus \partial \B} \|d\tilde{\rho}_j/d\rho_j\|=1$}. So overall, we obtain the bound (where from now on $a\lesssim b$ means $a\le C b$ for some constant $C$  depending on the dimension):
$$ \E((h^\B,\tilde{\rho}_j)^4)^{1/4}\lesssim (\frac{\delta}{a_j}\wedge 1)\sqrt{s(a_j)}. $$
Thus applying Cauchy--Schwarz to \eqref{eq:k4rewrite} we see that 
$$k_4(z_1,z_2,z_3,z_4)\lesssim \prod_{j=1}^4 (\frac{\delta}{a_j}\wedge 1) \sqrt{s(a_j)}.$$ 

On the other hand, we can lower bound $g(z_1,z_2,z_3,z_4)$ using the explicit expression $$G^\B(x,y)=s(|x-y|)-s(|x||y-\tilde{x}|) \quad ; \quad  \tilde{x}=|x|^{-2}x.$$ 

In particular, when $|x|=|y|=1-\delta=r$, we have that $|x|^2|y-\tilde{x}|^2|x-y|^{-2}=(1+|x-y|^{-2}\delta^2(2-\delta)^2)$. This implies that, on the region $|x-y|>\delta$, we have $ G^\B(x,y)\gtrsim \delta^2|x-y|^{-2}$  for $d=2$ and $G^\B(x,y) \gtrsim |x-y|^{2-d}\delta^2|x-y|^{-2}$ for $d\ge 3$. On the region $|x-y|<\delta$ we have $G^\B(x,y)\gtrsim s(|x-y|)$. 

Without loss of generality we can now assume that $a_1\ge a_2\ge a_3\ge a_4$. 
Combining the above lower bounds for $G^\B$ and the definition of $g$ with the upper bounds for $k_4$, we obtain \eqref{K4bound} under the condition that $a_2>a_1/10$. Note that the $\delta^{-\eta}$ correction is only actually needed when the dimension $d=2$. 

When $a_2 \leq a_1/10$, i.e. one point is considerably further than the rest, our bounds do not suffice - this is because we haven't taken properly care of cancellations occurring in the {third moment, when three points are together}. Let us do that now. 
Notice that when $a_2 \leq a_1/10$ we have $B_2,B_3,B_4\subset z_2+a_1\B$; we write $\hat\rho_i$ for harmonic measure seen from $z_2,z_3,z_4$ on $\partial(z_2+a_1\B)\setminus \partial\B$. We need an extension of \cref{claim:4mt}, that first separates the three points, and looks at the occurring cancellations:

\begin{claim}
In the case that $a_2<a_1/10$ we can further write $k_4(z_1,z_2,z_3,z_4)$ as 
\begin{equation}\label{eq:k4rewrite2}\E((h^\B,\tilde{\rho}_1)\prod_{i\ne 1} (h^\B,\hat\rho_i))+\sum_{\sigma\in S({2,3,4})} \E((h^\B,\tilde{\rho}_1)(h^\B,\hat\rho_{\sigma(2)}))\E((h^\B,\hat\rho_{\sigma(3)}-\tilde{\rho}_{\sigma(3)})(h^\B,\hat\rho_{\sigma(4)}-\tilde{\rho}_{\sigma(4)})).\end{equation}
\end{claim}

Again, we postpone the proof of the claim and first see how it implies \eqref{K4bound}. 
From the same Cauchy--Schwarz argument and bounds as in $a_2>a_1/10$ case (noting also that $d(z_j,\partial (z_2+a_1\B))>a_1/2$ for $j=3,4$), we can deduce that the first term in \eqref{eq:k4rewrite2} is $\lesssim \delta^{-\eta} g(z_1,z_2,z_3,z_4)$ for some $\eta\in [0,1)$. Again the $\delta^{-\eta}$ correction is only needed when $d=2$.

To deal with the latter terms in \eqref{eq:k4rewrite2}, we use the fact that the covariance of $h^\B$ is equal to $bG^\B$. Indeed, using the explicit expression for $G^\B$ we see that the latter term is given
by $$4b^2\sum_{\sigma\in S({2,3,4})}G^\B(z_1,z_{\sigma(2)})G^{(z_2+a_1\B)\cap \B}(z_{\sigma(3)},z_{\sigma(4)}),$$
and since $G^{D'}\le G^D$ for $D'\subset D$, \eqref{K4bound} then follows in the case $a_2 \leq a_1/10$ too. \medskip

{It now remains to prove the claims.}

\begin{proof}[Proof of Claims:]
We start with \eqref{eq:k4rewrite}. As mentioned above, we cannot apply the Markov property directly inside the regions $B_j$. Instead we will work in $2\B$ and use the equivalent of \eqref{eq:k4rewrite} for the field $h^{2\B}$, together with the Markov decomposition $h^{2\B}=h^\B_{2\B}+\varphi^\B_{2\B}$. In this decomposition, the two summands are independent, $h^{\B}_{2\B}$ has the law of $h^\B$ and $\varphi_{2\B}^\B$ is harmonic inside $\B$.

For each $i = 1 \dots 4$, let $\eta_i$ be a smooth function supported in $B_i$, rotationally symmetric about $z_i$, such that $k_4(z_1,z_2,z_3,z_4)=\E(\prod_i (h^\B,\eta_i))$. Let us also denote by $\rho_i'$ the  harmonic measure on $\partial B_i$ (now including the part on $\partial \B$) seen from $z_i$

By the Markov decomposition for the field $h^{2\B}$ inside $\{z_1+a_1\B,\dots, z_4+a_4\B\}$, as in \eqref{eq:multidmp} with $n=4$, we have that {$\E(\prod_{i=1}^4 (h^{2\B},\eta_i)) = \E(\prod_{i=1}^4 (h^{2\B},\rho_i')).$} We use the domain Markov decomposition $h^{2\B}=h^\B_{2\B}+\varphi^\B_{2\B}$ to rewrite this as
$$\E(\prod_{i=1}^4 (h^\B_{2\B}+\varphi^\B_{2\B},\eta_i)) = \E(\prod_{i=1}^4 (h^\B_{2\B}+\varphi^\B_{2\B},\rho_i')).$$
Opening the brackets and using the independence of $h^\B_{2\B}$ and $\varphi^\B_{2\B}$ we can write this further as
$$\sum_{S \cup S' = \{1,2,3,4\}} \E(\prod_{i\in S} (h^\B_{2\B},\eta_i))\E(\prod_{j\in S'} (\varphi^\B_{2\B},\eta_j)) = \sum_{S \cup S' = \{1,2,3,4\}} \E(\prod_{i\in S} (h^\B_{2\B},\rho_i'))\E(\prod_{j\in S'} (\varphi^\B_{2\B},\rho_j')).$$
As $h^{\B}_{2\B}$ has the law of $h^\B$, to prove that $\E(\prod_{i=1}^4 (h^{\B},\eta_i)) = \E(\prod_{i=1}^4 (h^{\B},\rho_i'))$ it suffices to show that all terms with $|S| \neq 4$ cancel out. Since $(h^\B,\tilde{\rho}_i)=(h^\B,\rho_i')$ a.s.\ for each $i$, this proves the claim.

To show the cancellation, first observe that as $\varphi^\B_{2\B}$ is harmonic inside each $B_j$, we have that $(\varphi^\B_{2\B},\eta_j) = (\varphi^\B_{2\B},\rho_j')$ for every $j = 1 \dots 4$ and thus the terms with $|S|' = 4$ cancel out. Also, both $h^\B_{2\B}$ and $\varphi^\B_{2\B}$ are mean zero, so all terms with $|S| = 1$ or $|S|' = 1$ also cancel out. We are left to consider the cases when $|S| = 2$ and $|S'| = 2$.
But we already know that $$\E(\prod_{j\in S'} (\varphi^\B_{2\B},\eta_j)) = \E(\prod_{j\in S'} (\varphi^\B_{2\B},\rho_j')),$$ and thus it remains to just verify that for $i \neq j$, we have that $\E((h^\B_{2\B},\eta_i) (h^\B_{2\B},\eta_j)) = \E((h^\B_{2\B},\rho_i') (h^\B_{2\B},\rho_j'))$. This can be however verified directly via the fact that the covariance is a multiple of the Green's function and all $\eta_i, \eta_j, \rho_i',\rho_j'$ have disjoint support.

The proof of \eqref{eq:k4rewrite2} given \eqref{eq:k4rewrite} follows from the same argument, using the domain Markov decomposition of $h^{2\B}$ inside $z_1+a_1\B$ and $z_2+a_1\B$. We omit the details.
\end{proof}
\end{proof}
\subsection{Gaussianity of spherical averages} 

We now show that the $r$-spherical average process around $0$ with a varying radius is a Gaussian process.

\begin{lemma}\label{lem:ca_bm}
$(h_{r}(0))_{r\in(0,1)}$ is a Gaussian process.
\end{lemma}

To prove this, we will use the fact that any continuous stochastic process (indexed by positive time) with independent increments, is Gaussian. The fact that $h_r(0)$ has a continuous modification in $r$
comes from \cref{lem:4mtbnd} of the previous subsection. 

\begin{proof}[Proof of \cref{lem:ca_bm}]
Since the variance of this process increases as $r\downarrow 0$, it is natural to parameterise the time so that the process starts at $r = 1$. Thus we set $X_t = h_{1-r}(0)$. Then $X_t\to 0$ in probability as $t\downarrow 0$ and $X$ has independent increments by the domain Markov property (more precisely \eqref{eq:it_dmp}). Thus, to show that it is a Gaussian process, by \cite[Theorem 11.4]{Kallenberg}, it suffices to prove that it admits a continuous modification in $t$. 

We apply Kolmogorov's continuity criterion to show this. By the domain Markov property and scaling assumption it suffices to control the behavior at time $0$, i.e. it is enough to show that for some $C\in (0,\infty)$, some $\eta<1$ and all $\delta\in (0,1)$:
\begin{equation}\label{eqn:required_4mom}
\E(X_\delta^4)\le C\delta^{2-\eta}.
\end{equation}   
For this we use \eqref{K4bound}. We obtain that
$$\E(X_\delta^4)=\iint_{\partial((1-\delta)\B)^4} k_4(z_1,z_2,z_3,z_4) \prod_{i=1}^4 \rho_0^{1-\delta}(dz_i) \lesssim \delta^{-\eta}\iint_{\partial((1-\delta)\B)^4} g(z_1,z_2,z_3,z_4) \prod_{i=1}^4 \rho_0^{1-\delta}(dz_i).$$ 
Now recalling that $g(z_1,z_2,z_3,z_4)$ is the four-point function for the zero boundary GFF in $\B$, we see that the integral on the right hand side is the fourth moment of the spherical average of the GFF at radius $1-\delta$. But this spherical average is a centred Gaussian with variance $-\log(1-\delta)$ when $d=2$ and $1-(1-\delta)^{2-d}$ when $d>2$ - see \cite[equation (13)]{WPgff}. Since these are both of order $\delta$ as $\delta\to 0$ the fourth moment is of order $\delta^2$ by Gaussianity {of the GFF itself}.  
\end{proof}

\subsection{Gaussianity in the general case}

\noindent In what follows, we will often use the co-ordinates $z=(r,\ol{\theta})=(|z|,z/|z|)$ for a point in $\R^d, d\ge 2$.

{We will generalise {the case of spherical averages} to a wider class of processes, stemming from so called spherical harmonics. The interest comes from the following classical theorem (see e.g. \cite[Chapter IV]{SteinWeiss}).

\begin{theorem}[Expansion using spherical harmonics]\label{thm:sph}
In each $d \geq 2$, there is a collection of smooth functions $\psi_{n,j}(\ol{\theta}): \partial \Ub \to \R$ with $n \in \N_0, M_n \in \N$ and $j \in \{1, \dots, M_n\} $ such that
\begin{enumerate}
    \item The functions $\psi_{n,j}(\ol{\theta})$ form an orthonormal basis of $L^2(\partial \Ub)$;
    \item For every $n \in \N_0$ and $j \in \{1, \dots, M_n\} $, we have that $(r,\ol{\theta})\mapsto r^n\psi_{n,j}(\ol{\theta})$ is harmonic in $\Ub$;
    \item For each $n \in \N_0$, one can find radially symmetric functions $f_{n,i}(r):[0,1] \to \R$ with $i \in \N_0$ such that $(e_{n,j,i} = f_{n,i}\psi_{n,j})_{n \in \N_0, j \in \{1, \dots, M_n\},i \in \N}$ form an orthonormal basis of $L^2(\Ub)$.
\end{enumerate}
\end{theorem}

\begin{rmk}
In fact one can write out a specific collection of such functions using Legendre polynomials and Bessel functions and choose $e_{n,j,i}$ to be eigenfunctions of $\Delta$. This is not necessary here. 

For example, in $d = 2$, one has $M_0=1$ and $M_n = 2$ for $n\ge 1$, with the $\psi$-s given by the usual Fourier series on the circle. That is, $$\{\{e_{0,1,k}\}_{k\ge 1},\{e_{n,1,k},e_{n,2,k}\}_{n\ge 1,k\ge 1}\}:=\{\{J_0(\alpha_{0,k}{r})\}_{k\ge 1}, \{J_n(\alpha_{n,k}r) \sin(n\theta), J_n(\alpha_{n,k}r) \cos(n\theta)\}_{n\ge 1, k\ge 1}\}$$ form an orthonormal basis of $L^2(\D)$, where $(J_n)_{n\ge 0}$ are the Bessel functions and $\alpha_{n,k}$ are the zeroes of $J_n$ for each $n$. 
\end{rmk}
Using these notations, the main result of this section can be stated as follows:

\begin{prop}[Gaussianity]\label{prop:gauss}
The random variables $(h^\D,e_{n,j,i})_{n \in \N, j \in \{1, \dots, M_n\},i \in \N}$ are jointly Gaussian. In particular $(h^\D, f)_{f \in C_c^\infty(\R^d)}$ is a Gaussian process.
\end{prop}

To prove \cref{prop:gauss} we will choose appropriate radial functions $g_{n,j}$ for which we can verify that $h^\B$ tested against $ g_{n,j}(r)\psi_{n,j}(\ol{\theta})$ on each sphere at radius $r$ is a Gaussian process in $r\in(0,1)$. The key observation is the following. For $r\in (0,1)$ and a smooth function $\psi: \partial \Ub \to \R$ let $\nu^\psi_r$ be the signed measure defined by the condition that
\begin{equation}\label{eq:av}
\nu^\psi_r(\phi)=\int_{\partial \Ub} \psi(\ol{\theta}) \phi(r,\ol{\theta}) \rho_0^1(d\ol{\theta})
\end{equation}
for all functions $\phi$ such that $\phi(r,\ol{\theta})\in L^1(\partial \Ub,\rho_0^1(d\ol{\theta}))$, where as before $\rho_0^1$ is uniform measure on $\partial \B$.

\begin{lemma}\label{lem:sin_const}
Suppose that $\varphi$ is a harmonic function in $\Ub$. Suppose also that $\psi(\ol{\theta})$ is a smooth function defined on $\partial \Ub$ such that $(r,\ol{\theta})\mapsto \psi(\ol{\theta})r^{n}$ is harmonic in $\Ub$. Then $r^{-n}\nu^\psi_r(\varphi)$ is constant {as a function of $r$ on $(0,1)$}.
\end{lemma}

\begin{proof}
Let us fix $r_0\in (0,1)$: we will show that $(\frac{d}{dr})[r^{-n}\nu_{r}^\psi(\varphi)]|_{r=r_0}=0$, which implies the result. By the second Green's identity applied in $r_0\Ub$ to the harmonic functions $\varphi$ and  $(r,\ol{\theta})\mapsto r^n\psi$, we can write 
$$\int_{\partial (r_0\Ub)}(\varphi\frac{d}{dr}[r^n\psi] - r^n\psi\frac{d}{dr}\varphi) = \int_{r_0 \Ub} (\varphi \Delta[ r^n\psi] - r^n\psi \Delta \varphi),$$
where we are integrating against the standard volume measure on $r_0\B$ on the right-hand side, and the induced measure on its boundary on the left, which is a multiple of uniform measure.
Now the right-hand side is zero as both $r^n\psi$ and $\varphi$ are harmonic by assumption. Thus we deduce that
$$n\int_{\partial \Ub}\varphi(r_0,\ol{\theta})\psi(\ol{\theta}) \, \rho_0^1(d\ol{\theta})= r_0\int_{\partial \Ub}\psi(\ol{\theta})\frac{d}{dr}\varphi(r,\ol{\theta})\big|_{r=r_0} \, \rho_0^1(d\ol{\theta}).$$
It follows that
$$\left.\frac{d}{dr}[r^{-n}\nu^\psi_r(\varphi)]\right|_{r=r_0} = -nr_0^{-n-1}\int_{\partial \Ub}\psi(\ol{\theta}) \varphi(r_0,\ol{\theta})\, \rho_0^1(d\ol{\theta})+ r_0^{-n}\int_{\partial \Ub}\psi(\ol{\theta})\frac{d}{dr}\varphi(r,\ol\theta)\big|_{r=r_0}\, \rho_0^1(d\ol{\theta}) = 0.$$
\end{proof}

Notice that all the functions $\psi_{n,j}$ of \cref{thm:sph} satisfy the conditions for the function $\psi$ in this lemma. A similar proof to the spherical average case now implies that any finite linear combination of the functions $(r^{-n}\nu^{\psi_{n,j}}_r)_{n \geq 0, j \in \{1, \dots, M_n\}}$ tested against $h^\B$ gives rise to a Gaussian process. 

\begin{lemma}\label{lem:gp} The process $(A_r)_{r\in (0,1)}$ defined by 
\begin{equation}\label{eq:gp}
A_r := \sum_{i=1}^k a_i (h^\D,r^{-n_i}\nu^{\psi_{n_i,j_i}}_r)
\end{equation} is Gaussian for any $k\ge 0$, $n_i\in \N_0,j_i \in \{1, \dots, M_{n_i}\}$ and $(a_1,\dots, a_k)$ real. 

In particular, for any fixed radius $r \in (0,1)$, we have that $(h^\B, \nu^{\psi_{n,j}}_r)_{n \in \N_0, j \in \{1, \dots, M_n\}}$ is jointly Gaussian.
\end{lemma}

}

\begin{proof} 
First let us fix $a_0,\dots, a_k$ in $\R$. We again parametrise the process from large radii towards small radii, i.e. let's set $X_t = A_{1-t}$. We first argue using \cref{lem:sin_const} that $X_t$ has independent increments. 

For clarity, let us first show this for $\tilde X_t = (h^\D,(1-t)^{-n_i}\nu^{\psi_{n_i,j_i}}_{1-t})$. Indeed, for each $i \in \{1, \dots, k\}$ and each $s < r < 1$ we can write by domain Markov property
$$(h^\D,s^{-n_i}\nu^{\psi_{n_i,j_i}}_s) = (h^{r\D},s^{-n_i}\nu^{\psi_{n_i,j_i}}_s) + (\varphi^{r\D},s^{-n_i}\nu^{\psi_{n_i,j_i}}_s).$$ 
 As already mentioned, $\psi_{n_i,j_i}$ satisfies the conditions of \cref{lem:sin_const}, and thus 
$$(\varphi^{r\D},s^{-n_i}\nu^{\psi_{n_i,j_i}}_s) = \lim_{u\uparrow r}(\varphi^{u\D},u^{-n_i}\nu^{\psi_{n_i,j_i}}_u)=(h^\B,r^{-n_i}\nu_r^{\psi_{n_i,j_i}})$$
where the last equality follows since  $(h^{r\B},u^{-n_i}\nu^{\psi_{n_i,j_i}}_u)$ converges to $0$ as $u\uparrow r$ in $L^2$, by a direct calculation using the Green's function.

As $h^{r\D}$ and $\varphi^{r\D}$ are independent, we conclude that {$\tilde X_t = (h^\D,(1-t)^{-n_i}\nu^{\psi_{n_i,j_i}}_{1-t})$ has independent increments.} {The same} argument can be directly applied to get the claim for $X_t$.

{Now, $X_t$ is also centred and square integrable, converging to $0$ in probability as $t\downarrow 0$}, by definition and assumptions on $h^\B$. Moreover, {using the fact that each $\psi_{n,j}$ is bounded on $\partial \B$, we can use the 4th moment bound \eqref{K4bound} to apply Kolmogorov's criteria similarly to the case of the spherical average and obtain that $X_t$ possesses an a.s. continuous modification.} This implies that the process is Gaussian by \cite[Theorem 11.4]{Kallenberg}.

For the final claim, let us fix $r\in (0,1)$. Then $(h^\B, \nu^{\psi_{n,j}}_r)_{n \geq 0, j \in \{1, \dots, M_n\}}$ is jointly Gaussian, since for any $k\ge 1$, $b_1,\dots, b_k\in \R$, $n_i\in \N_0$ and $j_i\in \{1,\dots, M_{n_i}\}$, we can apply the above with $a_i=r^{n_i}b_i$ for each $i$ to see that $\sum_i b_i (h^\B,\nu_r^{\psi_{n_i,j_i}})$ is a Gaussian random variable.
\end{proof}

There is one more step required to deduce \cref{prop:gauss}: to extend the joint Gaussianity to varying radii.

\begin{lemma}
For any $k \geq 1$,   and any $(r_i,n_i,j_i)_{1\le i \le k}$ with $r_i\in (0,1), n_i\in \N_0$ and $j_i\in \{1,\dots, M_{n_i}\}$ for each $i$, the vector $$\left((h^\B, \nu^{\psi_{n_1,j_1}}_{r_1}),\dots, (h^\B, \nu^{\psi_{n_k,j_k}}_{r_k})\right)$$ is Gaussian. 

In particular, we have that $(h^\B, \nu^{\psi_{n,j}}_r)_{r \in (0,1), n \geq 0, j \in \{1, \dots, M_n\}}$ is jointly Gaussian.

\end{lemma}

\begin{proof}
The second statement is an immediate consequence of the first. So, let us fix $k,(r_i,n_i,j_i)_{i\le k}$ as in the statement. 
For now, let us assume that $r_1>\dots>r_k>0$ (we come back to the general case at the end of the proof). We iterate the domain Markov property inside each ball $r_i\B$ as follows: we write $h^\B=h^0$, then $h^0=h^1+\varphi^1$ by applying the Markov property inside $r_1\B$, then $h^1=h^2+\varphi^2$ inside $r_2\B$ etc.\ so that $(\varphi^1,\dots, \varphi^k)$ are mutually independent by \eqref{eq:it_dmp}. Shortening the notation $\nu_r^{\psi_{n_i,j_i}}$ to $\nu_r^i$ for each $i=1,\dots, k$, it suffices to show that \begin{equation}\label{dmp_jointcircles}\left((h^0,\nu_{r_1}^1),(\varphi^1,\nu_{r_2}^2),(h^1,\nu_{r_2}^2),\dots, (\varphi^{k-1},\nu_{r_k}^k),(h^{k-1},\nu_{r_k}^k)\right)\end{equation} is a Gaussian vector - indeed, we can write each $(h^\B, \nu_{r_i}^i) = \sum_{j =1}^i (\varphi^j, \nu_{r_i}^i)$. {However, as in the proof \cref{lem:gp}, it follows from \cref{lem:sin_const}} 
that  $$(h^{j-1},\nu_{r_j}^{\psi_{n,i}})=(\varphi^j,\nu_{r_j}^{\psi_{n,i}})=(\frac{r_{j+1}}{r_j})^{n}(\varphi^j,\nu_{r_{j+1}}^{\psi_{n,i}})$$ for any $\psi_{n,i}$ and any $j$.
Hence, \eqref{dmp_jointcircles} can be rewritten as 
\begin{equation*}\big((\varphi^1,\nu_{r_1}^1),(\frac{r_2}{r_1})^{n_2}(\varphi^1,\nu_{r_1}^2),(\varphi^2,\nu_{r_2}^2),(\frac{r_3}{r_2})^{n_3}(\varphi^2,\nu_{r_2}^3),\dots,(\varphi^{k-1},\nu_{r_{k-1}}^{k-1}),(\frac{r_k}{r_{k-1}})^{n_k}(\varphi^{k-1},\nu_{r_{k-1}}^k),(\varphi^k,\nu_{r_k}^k)\big).\end{equation*}
{But now  $(\varphi^k,\nu_{r_k}^k)$ is Gaussian, and each pair $((\varphi^i,\nu_{r_i}^i),(\frac{r_{i+1}}{r_i})^{n_{i+1}}(\varphi^i,\nu_{r_i}^{i+1}))$ with $i = 1 \dots k -1$ is jointly Gaussian by \cref{lem:gp}. Moreover, this singleton, and each pair in the list are independent of one another by construction.} Thus we indeed have a Gaussian vector. 

Finally, if the $r_i$'s are not distinct, the same argument still works: the ``pairs'' just mentioned will simply be larger tuples, concluding the proof of the lemma.
\end{proof}

{We are now ready to prove  \cref{prop:gauss}.}

\begin{proof}[Proof of \cref{prop:gauss}]

Consider the basis {$(e_{n,j,i} = f_{n,i}\psi_{n,j})_{n \in \N_0, j \in \{1, \dots, M_n\},i \in \N}$} of $L^2(\B)$ given by \cref{thm:sph}. The previous lemma implies that 
\begin{equation}\label{eq:joint_manycircles}
\{(h^\B,  f_{n,i}(r)\nu_r^{\psi_{n,j}})\}_{r \in (0,1), n \in \N_0, j \in \{1, \dots, M_n\}, i \in \N}
\end{equation} 
is a Gaussian process. In particular,  
\begin{equation}\label{eq:finally_gaussian}
{(h^\B,f_{n,i}\psi_{n,j})_{n \in \N_0, j \in \{1, \dots, M_n\}, i \in \N}}
\end{equation} 
is also a Gaussian process. Indeed, the random variables in \eqref{eq:finally_gaussian} exist by \cref{cor}, and  we obtain the Gaussianity since they can be defined as almost sure limits of weighted sums of elements in the collection \eqref{eq:joint_manycircles}.

 The fact that $e_{n,j,i}$ form a basis of $L^2(\B)$, together with linearity of $(h^\B, f)$ and that $h^\B$ is zero outside of $\B$, now implies that $(h^\B,f)_{f \in C_c^\infty(\R^d)}$ is a Gaussian process.
 \end{proof}

\bibliographystyle{alpha}
\bibliography{gff_char}
\end{document}